\documentclass[12pt,reqno]{amsart}

\usepackage{amssymb}
\usepackage[all]{xy}

\usepackage{hyperref}

\numberwithin{equation}{section}

\newtheorem{theorem}{Theorem}[section]
\newtheorem{proposition}[theorem]{Proposition}
\newtheorem{lemma}[theorem]{Lemma}
\newtheorem{corollary}[theorem]{Corollary}

\theoremstyle{definition}
\newtheorem{definition}[theorem]{Definition}
\newtheorem{remark}[theorem]{Remark}
\newtheorem{example}[theorem]{Example}

\begin{document}

\baselineskip=15pt

\title[Lie algebroid connection and Harder-Narasimhan reduction]{Lie algebroid connection
and Harder-Narasimhan reduction}

\author[A. Bansal]{Ashima Bansal}

\address{Department of Mathematics, Shiv Nadar University, NH91, Tehsil Dadri,
Greater Noida, Uttar Pradesh 201314, India}

\email{ashima.bansal@snu.edu.in}

\author[I. Biswas]{Indranil Biswas}

\address{Department of Mathematics, Shiv Nadar University, NH91, Tehsil Dadri,
Greater Noida, Uttar Pradesh 201314, India}

\email{indranil.biswas@snu.edu.in, indranil29@gmail.com}

\author[P. Kumar]{Pradip Kumar}

\address{Department of Mathematics, Shiv Nadar University, NH91, Tehsil Dadri,
Greater Noida, Uttar Pradesh 201314, India}

\email{Pradip.Kumar@snu.edu.in}



\date{}

\begin{abstract}
Take a holomorphic Lie algebroid $(V,\, \phi)$ on a compact connected Riemann surface $X$
such that the anchor map $\phi$ is not surjective. Let $P$ be a parabolic subgroup of a
complex reductive affine algebraic group $G$ and $E_P\, \subset\, E_G$ a holomorphic reduction
of structure group, to $P$, of a holomorphic principal $G$--bundle $E_G$ on $X$. We prove that
$E_P$ admits a holomorphic Lie algebroid connection for $(V,\,\phi)$ if the reduction $E_P$ is
infinitesimally rigid. If $E_P$ is the Harder--Narasimhan reduction of $E_G$, then it is shown that
$E_P$ admits a holomorphic Lie algebroid connection for $(V,\,\phi)$. In particular, for any
point $x_0\,\in\, X$, the Harder--Narasimhan reduction $E_P$ admits a logarithmic connection that
is nonsingular on the complement $X\setminus\{x_0\}$.
\end{abstract} 

\maketitle

\section{Introduction}

An indecomposable holomorphic vector bundle $E$ on a compact connected Riemann surface $X$ admits a
holomorphic connection if and only if the degree of $E$ is zero; more generally,
a holomorphic vector bundle $E$ on $X$ admits a holomorphic connection 
if and only if the degree of each indecomposable
component of $E$ is zero \cite{At}. This criterion for holomorphic connections generalizes 
to holomorphic principal $G$--bundles on $X$, where $G$ is a complex reductive affine algebraic group
\cite{AB}.

Lie algebroid connections are generalizations of holomorphic --- and also logarithmic --- connections.
A holomorphic Lie algebroid on $X$ is a pair $(V,\,\phi)$, where $V$ is a holomorphic vector bundle on $X$
equipped with the structure of a $\mathbb{C}$--bilinear Lie algebra on its sheaf of holomorphic sections, and
$\phi\,:\, V\, \longrightarrow\, TX$ is an ${\mathcal O}_X$--linear homomorphism satisfying the Leibniz rule
$$
[s,\, f t] \ =\ f\,[s,\, t]\;+\;\phi(s)(f)\,t
$$
for all locally defined holomorphic sections $s,\,t$ of $V$ and all locally defined holomorphic
functions $f$ on $X$. Lie algebroid connections are defined by replacing the tangent bundle $TX$ by $V$.
The map $\phi$ is used in formulating the Leibniz rule for connections. See \cite{BMRT}, \cite{BR},
\cite{To1}, \cite{To2} and \cite{AO} for Lie algebroids and Lie algebroid connections.

Here we address the following question. Let $E$ be a holomorphic vector bundle on $X$
of rank $r$ equipped with a filtration of holomorphic subbundles
$$
E_1\ \subset\ E_2\ \subset\ \cdots\ \subset\ E_{n-1}\ \subset\ E_n\ =\ E.
$$
Is there a holomorphic Lie algebroid connection $E$ that preserves every $E_i$, $1\, \leq\, i\, \leq\, n$?
This question is cast in the more general framework of holomorphic principal $G$--bundles $E_G$ equipped
with a holomorphic reduction of structure group $E_P\, \subset\, G$, where $G$ is a complex
reductive affine algebraic group and $P\, \subset\, G$ is a parabolic subgroup. When $G\,=\, \text{GL}(r,
{\mathbb C})$, the pair $E_P\, \subset\, E_G$ corresponds to the above data of a holomorphic vector bundle
$E$ of rank $r$ equipped with a filtration of holomorphic subbundles.

Take a holomorphic reduction of structure group $E_P\, \subset\, E_G$. The tangent space --- at this reduction
$E_P$ --- to the space of all holomorphic reductions of structure group of $E_G$ to $P$ is given by
$H^0(X,\, \text{ad}(E_G)/\text{ad}(E_P))$, where $\text{ad}(E_G)$ and $\text{ad}(E_P)$ are the adjoint vector
bundles for the principal bundles $E_G$ and $E_P$ respectively. The holomorphic reduction of structure group
$E_P\, \subset\, E_G$ is called infinitesimally rigid if
$$
H^0(X,\, \text{ad}(E_G)/\text{ad}(E_P))\ =\ 0
$$
(see Definition \ref{def2}). An example of infinitesimally rigid reduction of structure group
of $E_G$ is given by the Harder--Narasimhan reduction of $E_G$ (see Lemma \ref{lem1}).

We prove the following (see Theorem \ref{thm1}):

\begin{theorem}\label{thm0}
Let $(V,\, \phi)$ be a holomorphic Lie algebroid over $X$ satisfying the condition that the
anchor map $\phi$ is not surjective. Let $E_P\, \subset\, E_G$ be a holomorphic reduction of
structure group which is infinitesimally rigid. Then the holomorphic
principal $P$--bundle $E_P$ admits a holomorphic Lie algebroid connection for $(V,\,\phi)$.
\end{theorem}

Using Theorem \ref{thm0} the following is proved (see Corollary \ref{cor2}):

\begin{corollary}\label{cor00}
Take a point $x_0\, \in\, X$. Let $E_P\, \subset\, E_G$ be the Harder--Narasimhan reduction of $E_G$.
Then the holomorphic principal $P$--bundle $E_P$ admits a logarithmic connection which is nonsingular
over $X\setminus \{x_0\}$ (singularity allowed only over $x_0$).
\end{corollary}

We note that the Harder--Narasimhan reduction $E_P$ does not admit any holomorphic
connection (see Remark \ref{rem2}).

Theorem \ref{thm0} also gives the following (see Corollary \ref{cor3}):

\begin{corollary}\label{cor0}
Let $E$ be a holomorphic vector bundle on $X$, and let
$$
E_1\ \subset\ E_2\ \subset\ \cdots\ \subset\ E_{n-1}\ \subset\ E_n\ =\ E.
$$
be the Harder--Narasimhan filtration of $E$. Then the following two statements hold:
\begin{enumerate}
\item Take a point $x_0\, \in\, X$. The holomorphic vector bundle $E$ admits a logarithmic connection
$D$ satisfying the following two conditions:
\begin{itemize}
\item $D$ is nonsingular over $X\setminus \{x_0\}$, and

\item $D$ preserves $E_i$ for every $1\, \leq\, i\, \leq\, n$.
\end{itemize}

\item Let $(V,\, \phi)$ be a holomorphic Lie algebroid over $X$ satisfying the condition that the
anchor map $\phi$ is not surjective. Then $E$ admits a Lie algebroid connection $D$ for $(V,\, \phi)$
that preserves $E_i$ for every $1\, \leq\, i\, \leq\, n$.
\end{enumerate}
\end{corollary}

\section{Lie algebroids and Lie algebroid connections}

\subsection{Holomorphic Lie algebroids}

Let $X$ be a compact connected Riemann surface.
The holomorphic tangent bundle
of $X$ will be denoted by $TX$, while the holomorphic cotangent bundle
of $X$ will be denoted by $K_X$.
The first jet bundle of a holomorphic vector bundle
$W$ over $X$ will be denoted by $J^1(W)$.

A $\mathbb{C}$--Lie algebra structure on a holomorphic vector
bundle $V$ over $X$ is a $\mathbb{C}$--bilinear pairing defined by a sheaf homomorphism
$$
[-,\, -] \,\,:\,\, V\otimes_{\mathbb C} V \,\, \longrightarrow\,\, V,
$$
which is given by a holomorphic homomorphism $J^1(V)\otimes J^1(V)\, \longrightarrow\, V$
of vector bundles, such that
$[s,\, t]\,=\, -[t,\, s]$ and $[[s,\, t],\, u]+[[t,\, u],\, s]+[[u,\, s],\, t]\,=\,0$
for all locally defined holomorphic sections $s,\, t,\, u$ of $V$.
The Lie bracket operation on the sheaf of holomorphic vector fields on $X$ gives the structure of
a $\mathbb C$--Lie algebra on $TX$.

A holomorphic Lie algebroid over $X$ is a pair $(V,\, \phi)$, where
\begin{enumerate}
\item $V$ is a holomorphic vector bundle over $X$ equipped with the structure of a
$\mathbb{C}$--Lie algebra, and

\item $\phi\, :\, V\, \longrightarrow\, TX$ is an ${\mathcal O}_X$--linear homomorphism
such that
\begin{equation}\label{esc}
[s,\, f\cdot t]\ =\ f\cdot [s,\, t]+\phi(s)(f)\cdot t
\end{equation}
for all locally defined holomorphic sections
$s,\, t$ of $V$ and all locally defined holomorphic functions $f$ on $X$.
\end{enumerate}
The above homomorphism $\phi$ is called the \textit{anchor map} of the Lie algebroid.
The two conditions in the definition of a Lie algebroid imply that
\begin{equation}\label{e5}
\phi([s,\, t])\ =\ [\phi(s),\, \phi(t)]
\end{equation}
for all locally defined holomorphic sections $s,\, t$ of $V$.

\subsection{Lie algebroid connection}

Let $\mathbb G$ be a complex Lie group. Take a holomorphic
principal $\mathbb G$--bundle
\begin{equation}\label{f1}
p\ :\ E_{\mathbb G}\ \longrightarrow\ X.
\end{equation}
The action of $\mathbb G$
on $E_{\mathbb G}$ produces an action of $\mathbb G$ on the direct image $p_*TE_{\mathbb G}$ of
the holomorphic tangent bundle $TE_{\mathbb G}$. The invariant part
\begin{equation}\label{e6}
\psi\ :\ \text{At}(E_{\mathbb G})\ :=\ (p_*TE_{\mathbb G})^{\mathbb G} \ \longrightarrow\ X
\end{equation}
is the Atiyah bundle for $E_{\mathbb G}$ \cite{At}. It fits in a short exact sequence of holomorphic vector bundles
\begin{equation}\label{e7}
0\, \longrightarrow\, \text{ad}(E_{\mathbb G})\, \xrightarrow{\,\,\,\iota\,\,\,} \, \text{At}(E_{\mathbb G})
\, \xrightarrow{\,\,\,\varpi\,\,\,} \, TX\, \longrightarrow\, 0,
\end{equation}
where $\text{ad}(E_{\mathbb G})$ is the adjoint vector bundle for $E_{\mathbb G}$ (see \cite{At}); the projection
$\varpi$ in \eqref{e7} is given by the differential $dp\, :\, TE_{\mathbb G}\, \longrightarrow\, p^*TX$
of the map $p$ in \eqref{f1}. The sequence in \eqref{e7} is known as the Atiyah exact sequence for $E_{\mathbb G}$.

Let $(V,\, \phi)$ be a holomorphic Lie algebroid over $X$. Consider the homomorphism
$$
\psi\ :\ V\oplus \text{At}(E_{\mathbb G}) \ \longrightarrow\ TX, \ \ \, (v,\, w)\
\longmapsto\ \phi(v) - \varpi(w),
$$
where $\varpi$ is the homomorphism in \eqref{e7}. Define
\begin{equation}\label{e8}
\mathcal{A}(E_{\mathbb G}) \ :=\ \text{kernel}(\psi)\ \subset\ V\oplus \text{At}(E_{\mathbb G}).
\end{equation}
The $\mathbb C$--Lie algebra structure on $V\oplus \text{At}(E_{\mathbb G})$, given by the
$\mathbb C$--Lie algebra structures on $V$ and $\text{At}(E_{\mathbb G})$, restricts to a
$\mathbb C$--Lie algebra structure on $\mathcal{A}(E_{\mathbb G})$.
Restricting the natural projection $V\oplus \text{At}(E_{\mathbb G})\, \longrightarrow\, V$ to
$\mathcal{A}(E_{\mathbb G})\, \subset\, V\oplus \text{At}(E_{\mathbb G})$ we obtain a homomorphism
\begin{equation}\label{e9}
\rho\ :\ \mathcal{A}(E_{\mathbb G}) \ \longrightarrow\ V;
\end{equation}
note that $\text{kernel}(\rho) \,=\, \text{kernel}(\varpi)\,=\, \text{ad}(E_{\mathbb G})$.
Similarly, restricting the natural projection $V\oplus \text{At}(E_{\mathbb G})\, \longrightarrow\,
\text{At}(E_{\mathbb G})$
to $\mathcal{A}(E_{\mathbb G})\, \subset\, V\oplus \text{At}(E_{\mathbb G})$ we obtain a homomorphism
\begin{equation}\label{e10}
\varphi\ :\ \mathcal{A}(E_{\mathbb G}) \ \longrightarrow\ \text{At}(E_{\mathbb G}).
\end{equation}

We have the commutative diagram of homomorphisms of vector bundles
\begin{equation}\label{e11}
\begin{matrix}
0 &\longrightarrow & {\rm ad}(E_{\mathbb G}) & \longrightarrow & {\mathcal A}(E_{\mathbb G}) &
\xrightarrow{\,\,\,\rho\,\,\,} & V & \longrightarrow & 0\\
&& \Big\Vert &&\,\,\, \Big\downarrow\varphi &&\,\,\, \Big\downarrow \phi\\
0 & \longrightarrow & {\rm ad}(E_{\mathbb G}) & \stackrel{\iota}{\longrightarrow} &
{\rm At}(E_{\mathbb G})& \xrightarrow{\,\,\,\varpi\,\,\,} & TX & \longrightarrow & 0
\end{matrix}
\end{equation}
where $\varphi$ and $\rho$ are constructed in \eqref{e10} and \eqref{e9} respectively.

\begin{definition}\label{def1}
A \textit{holomorphic Lie algebroid connection} on $E_{\mathbb G}$ for $(V,\,\phi)$ is a
holomorphic homomorphism
$$
\delta\ : \ V\ \longrightarrow\ \mathcal{A}(E_{\mathbb G})
$$
such that $\rho\circ\delta\,=\, {\rm Id}_V$, where $\rho$ is the homomorphism in \eqref{e9}.
\end{definition}

When $V\,=\, TX$ and $\phi\,=\, {\rm Id}_{TX}$, a holomorphic Lie algebroid connection on
$E_{\mathbb G}$ is a usual holomorphic connection on the principal
$\mathbb G$--bundle $E_{\mathbb G}$.
When ${\mathbb G}\,=\, \text{GL}(r,{\mathbb C})$, the notion of a holomorphic Lie algebroid connection
on $E_{\mathbb G}$ coincides with that for the holomorphic vector bundle of rank $r$ on $X$ associated
to $E_{\mathbb G}$ for the standard representation of $\text{GL}(r,{\mathbb C})$.

\begin{lemma}\label{lem-1}
Giving a holomorphic Lie algebroid connection on $E_{\mathbb G}$ for $(V,\, \phi)$ is equivalent to giving
a holomorphic homomorphism
$$
\delta'\ :\ V\ \longrightarrow\ {\rm At}(E_G)
$$
such that $\varpi\circ\delta'\,=\, \phi$, where $\varpi$ is the projection in \eqref{e7}.
\end{lemma}

\begin{proof}
Take any holomorphic homomorphism $\delta'\, :\, V\, \longrightarrow\, {\rm At}(E_G)$
such that $\varpi\circ\delta'\,=\, \phi$. Consider the homomorphism
$\delta\, :\, V\, \longrightarrow\, \mathcal{A}(E_{\mathbb G})$ that sends any $v\, \in\, V$
to $(v,\, \delta'(v))\, \in\, \mathcal{A}(E_{\mathbb G})$; note that the given condition
that $\varpi\circ\delta'\,=\, \phi$ ensures that $(v,\, \delta'(v))\, \in\, \mathcal{A}(E_{\mathbb G})$. It
is evident that $\delta$ is a holomorphic Lie algebroid connection on $E_{\mathbb G}$ for $(V,\, \phi)$.

Conversely, if $\delta\, :\, V\, \longrightarrow\, \mathcal{A}(E_{\mathbb G})$
is a holomorphic Lie algebroid connection on $E_{\mathbb G}$ for $(V,\, \phi)$, then
$\delta'\,=\, \varphi\circ\delta$, where $\varphi$ is constructed in \eqref{e10}, is a homomorphism
$V\, \longrightarrow\, {\rm At}(E_G)$ that satisfies the condition $\varpi\circ\delta'\,=\, \phi$.
\end{proof}

\section{Rigid parabolic reduction of structure group}

Let $G$ be a connected reductive complex affine algebraic group. The Lie algebra of $G$ will be denoted
by $\mathfrak g$. Since $G$ is reductive, there are $G$--invariant nondegenerate symmetric bilinear forms
on $\mathfrak g$. Fix a $G$--invariant nondegenerate symmetric bilinear form
\begin{equation}\label{e1}
\textbf{B}\ \in\ \text{Sym}^2({\mathfrak g}^*)^G .
\end{equation}
This form $\textbf{B}$ produces an isomorphism of $G$--modules
\begin{equation}\label{e2}
{\mathfrak g}\ \xrightarrow{\,\,\,\, \sim\,\,\,}\ {\mathfrak g}^*.
\end{equation}

A Zariski closed connected subgroup of
$G$ is called parabolic if the quotient $G/P$ is a projective variety \cite[p.~183, Section 30]{Hu2},
\cite[\S~11.2]{Bo}. Fix a parabolic proper subgroup
\begin{equation}\label{e3}
P \ \subsetneq \ G.
\end{equation}
The Lie algebra of $P$ will be denoted by $\mathfrak p$. Let
\begin{equation}\label{e4}
R_n(\mathfrak p) \subset\ \mathfrak p
\end{equation}
be the nilpotent radical of $\mathfrak p$ \cite[Ch.~1, \S~3]{Hu1}; so $R_n(\mathfrak p)$ is the
Lie algebra of the unipotent radical of $P$ (see \cite[11.21, p.~157]{Bo}, \cite[p.~125]{Hu2} for
unipotent radical). Consider the orthogonal complement
${\mathfrak p}^\perp\, \subset\, {\mathfrak g}$ of ${\mathfrak p}\, \subset\, {\mathfrak g}$
for the pairing $\textbf{B}$ in \eqref{e1}. It is straightforward to check that
\begin{equation}\label{e12}
{\mathfrak p}^\perp\ =\ R_n(\mathfrak p).
\end{equation}
{}From \eqref{e12} it follows that
\begin{equation}\label{e13}
R_n(\mathfrak p)^* \ =\ {\mathfrak g}/{\mathfrak p}.
\end{equation}

As before, $X$ is a compact connected Riemann surface. Let $$p\, :\, E_G\, \longrightarrow\, X$$ be a
holomorphic principal $G$--bundle on $X$. As before, the adjoint vector bundle for $E_G$ is denoted
by $\text{ad}(E_G)$. Since $G/P$ (see \eqref{e3}) is compact, it follows that
$E_G$ admits holomorphic reductions of structure group to the subgroup $P\, \subset\, G$. Let
\begin{equation}\label{e14}
E_P\ \subset\ E_G
\end{equation}
be a holomorphic reduction of structure group of $E_G$ to $P$. Consider the adjoint vector bundle
$\text{ad}(E_P)\, \subset\, \text{ad}(E_G)$ for $E_P$. Let $${\mathcal R}(E_P)\ \subset\
\text{ad}(E_P)$$ be the nilpotent radical bundle; so for every point $x\, \in\, X$, the fiber
${\mathcal R}(E_P)_x$ of ${\mathcal R}(E_P)$ over the point $x$ is the nilpotent radical of the
fiber $\text{ad}(E_P)_x$. Note that $\text{ad}(E_P)_x$ is isomorphic to $\mathfrak p$ and ${\mathcal R}(E_P)_x$
is isomorphic to $R_n(\mathfrak p)$ (see \eqref{e4}). From \eqref{e13} it follows immediately that
\begin{equation}\label{e15}
{\mathcal R}(E_P)^*\ =\ \text{ad}(E_G)/\text{ad}(E_P).
\end{equation}

\begin{definition}\label{def2}
The holomorphic reduction of structure group $E_P\, \subset\, E_G$ in \eqref{e14} is called \textit{infinitesimally
rigid} if
$$
H^0(X,\, \text{ad}(E_G)/\text{ad}(E_P))\ =\ 0.
$$
{}From \eqref{e15} it follows that $H^0(X,\, \text{ad}(E_G)/\text{ad}(E_P))\, =\, 0$ if and only if
$H^0(X,\, {\mathcal R}(E_P)^*)\, =\, 0$.
\end{definition}

\begin{remark}\label{rem1}
Consider the space of holomorphic reductions of structure group of $E_G$ to $P$. The tangent space to
it at such a reduction $E_P\, \subset\, E_G$ is given by $H^0(X,\, \text{ad}(E_G)/\text{ad}(E_P))$. Therefore,
$E_P$ is infinitesimally rigid if and only if this tangent space at the point corresponding $E_P$ is the
zero vector space.
\end{remark}

\begin{example}\label{ex1}
Let $E_{\text{GL}(2, {\mathbb C})} \, \longrightarrow\, {\mathbb C}{\mathbb P}^1$ denote the
holomorphic principal $\text{GL}(2, {\mathbb C})$--bundle defined by ${\mathcal O}_{{\mathbb C}{\mathbb P}^1}\oplus
{\mathcal O}_{{\mathbb C}{\mathbb P}^1}(1)$. Let $P\, \subset\, \text{GL}(2, {\mathbb C})$ be the
parabolic subgroup given by all $A\,=\, (A_{ij})_{1\leq i,j\, \leq 2}\, \in\,
\text{GL}(2, {\mathbb C})$ such that $a_{21}\,=\, 0$.

If $E_P$ is the parabolic reduction of $E_{\text{GL}(2, {\mathbb C})}$ given by the filtration
${\mathcal O}_{{\mathbb C}{\mathbb P}^1}(1)\, \subset\, {\mathcal O}_{{\mathbb C}{\mathbb P}^1}\oplus
{\mathcal O}_{{\mathbb C}{\mathbb P}^1}(1)$. Then $\text{ad}(E_{\text{GL}(2, {\mathbb C})})/\text{ad}(E_P)\,=\,
{\mathcal O}_{{\mathbb C}{\mathbb P}^1}(-1)$. Hence the reduction $E_P$ is infinitesimally rigid.

If $E_P$ is the parabolic reduction of $E_{\text{GL}(2, {\mathbb C})}$ given by the filtration
${\mathcal O}_{{\mathbb C}{\mathbb P}^1}\, \subset\, {\mathcal O}_{{\mathbb C}{\mathbb P}^1}\oplus
{\mathcal O}_{{\mathbb C}{\mathbb P}^1}(1)$. Then $\text{ad}(E_{\text{GL}(2, {\mathbb C})})/\text{ad}(E_P)\,=\,
{\mathcal O}_{{\mathbb C}{\mathbb P}^1}(1)$. Hence the reduction $E_P$ is \textit{not} infinitesimally rigid.
\end{example}

As in \eqref{e14}, $E_G$ is a holomorphic principal $G$--bundle on a compact connected Riemann surface $X$.
We assume that $E_G$ is \textit{not} semistable. Let
$$
E_P\ \subset\ E_G
$$
be the Harder--Narasimhan reduction of $E_G$ (see \cite{AAB}).

\begin{lemma}\label{lem1}
The above Harder--Narasimhan reduction $E_P\, \subset\ E_G$ is infinitesimally rigid.
\end{lemma}

\begin{proof}
This follows directly from \cite[p.~705, Corollary 1]{AAB}.
\end{proof}

\section{Connection preserving a reduction of structure group}

As before, $X$ is a compact connected Riemann surface, $G$ is a connected reductive complex affine
algebraic group and $P\, \subsetneq\, G$ is a parabolic subgroup. Take a holomorphic principal
$G$--bundle $E_G$ on $X$ and a holomorphic reduction of structure group
\begin{equation}\label{e22}
E_P\ \subset\ E_G
\end{equation}
of $E_G$ to the subgroup $P$.

Let $R_u(P)\, \subset\, P$ be the unipotent radical \cite[p.~184, Section 30.2]{Hu2}. Let
\begin{equation}\label{e16}
L(P) \ :=\ P/R_u(P)
\end{equation}
be the Levi quotient of $P$. Consider
\begin{equation}\label{e17}
E_{L(P)} \ :=\ E_P/R_u(P)\ \longrightarrow\ X.
\end{equation}
Since $R_u(P)$ is a normal subgroup of $P$, the quotient $E_{L(P)}$ in \eqref{e17} coincides with
the holomorphic principal $L(P)$--bundle on $X$ obtained by extending the structure group of
$E_P$ using the quotient map $P\, \longrightarrow\, P/R_u(P)\,=\, L(P)$ (see \eqref{e16}).

\begin{proposition}\label{prop1}
Let $(V,\, \phi)$ be a holomorphic Lie algebroid over $X$ satisfying the condition that the
anchor map $\phi$ is not surjective. Then the holomorphic principal $L(P)$--bundle
$E_{L(P)}$ in \eqref{e17} admits a holomorphic Lie algebroid connection for $(V,\,\phi)$.
\end{proposition}

\begin{proof}
Note that the Levi quotient $L(P)$ is connected and reductive.
Since the anchor map $\phi$ is not surjective, the Lie algebroid $(V,\, \phi)$ is not split.
Now Theorem 1.1 of \cite{Bi} says that the holomorphic principal $L(P)$--bundle
$E_{L(P)}$ admits a holomorphic Lie algebroid
connection for $(V,\,\phi)$.
\end{proof}

\begin{theorem}\label{thm1}
Let $(V,\, \phi)$ be a holomorphic Lie algebroid over $X$ satisfying the condition that the
anchor map $\phi$ is not surjective. Assume that the holomorphic reduction of structure group
$E_P\, \subset\, E_G$ (see \eqref{e22}) is infinitesimally rigid. Then the holomorphic
principal $P$--bundle $E_P$ admits a holomorphic Lie algebroid connection for $(V,\,\phi)$.
\end{theorem}

\begin{proof}
The subsheaf $\phi(V)\, \subsetneq\, TX$ will be denoted by $\mathcal T$. So we have an inclusion map
\begin{equation}\label{e23}
\iota\ :\ {\mathcal T}\ = \ TX\otimes {\mathcal O}_X(-{\mathbb D})
\ \hookrightarrow\ TX,
\end{equation}
where $\mathbb D$ is a nonzero effective divisor. Consider the holomorphic Lie algebroid
\begin{equation}\label{e24}
{\mathcal L} \, :=\ ({\mathcal T},\, \iota)\ =\ (TX\otimes {\mathcal O}_X(-{\mathbb D}),\, \iota).
\end{equation}
If
$$
\delta'\ :\ {\mathcal T}\ \longrightarrow\ {\rm At}(E_P)
$$
is a holomorphic homomorphism giving a holomorphic Lie algebroid connection on
$E_P$ for the holomorphic Lie algebroid $\mathcal L$ in \eqref{e24} (see Lemma \ref{lem-1}), then
$$
\delta'\circ\phi\ :\ V \ \longrightarrow\ {\rm At}(E_P)
$$
is a holomorphic Lie algebroid connection on $E_P$ for the holomorphic Lie algebroid $(V,\, \phi)$.

Therefore, to prove the theorem it suffices to show that the holomorphic
principal $P$--bundle $E_P$ admits a holomorphic Lie algebroid connection for the holomorphic Lie
algebroid $\mathcal L$ in \eqref{e24}.

Consider the short exact sequences for the principal bundles $E_P$ and $E_{L(P)}$ (see \eqref{e17})
$$
0\ \longrightarrow\ {\rm ad}(E_P)\ \longrightarrow\ {\mathcal A}(E_P) \
\xrightarrow{\,\,\,\rho\,\,\,}\ {\mathcal T} \ \longrightarrow \ 0
$$
$$
0\ \longrightarrow\ {\rm ad}(E_{L(P)})\ \longrightarrow\ {\mathcal A}(E_{L(P)}) \
\xrightarrow{\,\,\,\delta\,\,\,}\ {\mathcal T} \ \longrightarrow \ 0
$$
(see \eqref{e11} and \eqref{e24}). Since the principal $L(P)$--bundle $E_{L(P)}$ is the extension of the structure
group of the principal $P$--bundle $E_P$ for the quotient map $P\, \longrightarrow\, L(P)$ (see \eqref{e16}),
the two exact sequences fit in the following commutative diagram:
\begin{equation}\label{e18}
\begin{matrix}
&& 0 && 0\\
&& \Big\downarrow && \Big\downarrow\\
&& {\mathcal R}(E_P) & = & {\mathcal R}(E_P)\\
&& \,\,\, \Big\downarrow\alpha' && \,\,\,\Big\downarrow\alpha\\
0 &\longrightarrow & {\rm ad}(E_P) & \longrightarrow & {\mathcal A}(E_P) &
\xrightarrow{\,\,\,\rho\,\,\,} & {\mathcal T} & \longrightarrow & 0\\
&& \,\,\, \Big\downarrow\beta &&\,\,\, \Big\downarrow\gamma && \Big\Vert\\
0 & \longrightarrow & {\rm ad}(E_{L(P)}) & \longrightarrow &
{\mathcal A}(E_{L(P)}) & \xrightarrow{\,\,\,\eta\,\,\,} & {\mathcal T} & \longrightarrow & 0\\
&& \Big\downarrow && \Big\downarrow\\
&& 0 && 0
\end{matrix}
\end{equation}
where ${\mathcal R}(E_P)\, \subset\, {\rm ad}(E_P)$ is the nilpotent radical bundle in \eqref{e15}, and
$\alpha$ is the composition of $\alpha'$ with the inclusion map ${\rm ad}(E_P)\, \hookrightarrow\,
{\mathcal A}(E_P)$. From Proposition \ref{prop1} we know that the principal $L(P)$--bundle $E_{L(P)}$
admits a holomorphic Lie algebroid connection for $\mathcal L$. Fix a holomorphic Lie algebroid connection
on $E_{L(P)}$ for $\mathcal L$; so $\delta$ is a holomorphic homomorphism
\begin{equation}\label{e19}
\delta\ :\ {\mathcal T} \ \longrightarrow\ {\mathcal A}(E_{L(P)})
\end{equation}
such that $\eta\circ\delta\,=\, {\rm Id}_{\mathcal T}$ (see Definition \ref{def1}).

Consider the subbundle ${\mathcal B}(\delta)\,:=\, \gamma^{-1}(\delta({\mathcal T}))\, \subset\, {\mathcal A}
(E_P)$, where $\gamma$ and $\delta$ are the homomorphisms in \eqref{e18} and \eqref{e19} respectively. From
\eqref{e18} we have the following short exact sequence:
\begin{equation}\label{e20}
0\ \longrightarrow\ {\mathcal R}(E_P)\ \xrightarrow{\,\,\,\alpha\,\,\, }\ {\mathcal B}(\delta)
\ \xrightarrow{\,\,\, \eta\circ\gamma\,\,\,} \ {\mathcal T} \ \longrightarrow\ 0,
\end{equation}
where $\alpha$ and $\eta$ are the homomorphisms in \eqref{e18}.

We will show the following: Any holomorphic homomorphism
$$
\theta\ :\ {\mathcal T} \ \longrightarrow\ {\mathcal B}(\delta)
$$
satisfying the condition $\eta\circ\gamma\circ\theta\,=\, {\rm Id}_{\mathcal T}$ defines a holomorphic Lie
algebroid connection on $E_P$ for $\mathcal L$. To prove this, consider the composition of homomorphisms
$$
{\mathcal T} \ \xrightarrow{\,\,\,\theta\,\,\,} \ {\mathcal B}(\delta)\ \hookrightarrow\ {\mathcal A}(E_P).
$$
Let $\theta'\, :\, {\mathcal T} \ \longrightarrow\ {\mathcal A}(E_P)$ denote this composition of
homomorphisms. Then from
the given condition that $\eta\circ\gamma\circ\theta\,=\, {\rm Id}_{\mathcal T}$, and the commutativity of
the diagram in \eqref{e18}, it follows immediately that
$\rho\circ\theta'\,=\, {\rm Id}_{\mathcal T}$. Hence $\theta'$ is a holomorphic Lie algebroid
connection on $E_P$ for $\mathcal L$ (see Definition \ref{def1}).

In other words, to prove the theorem it suffices to show that the short exact sequence in \eqref{e20}
splits holomorphically.

Let
\begin{equation}\label{e25}
\Psi\ \in\ H^1(X,\, \text{Hom}({\mathcal T},\, {\mathcal R}(E_P)))\ =\
H^1(X,\, {\mathcal R}(E_P)\otimes {\mathcal T}^*)
\end{equation}
be the cohomology class for the short exact sequence in \eqref{e20}. By Serre duality,
\begin{equation}\label{e26}
H^1(X,\, {\mathcal R}(E_P)\otimes {\mathcal T}^*)\ =\ H^0(X,\, {\mathcal R}(E_P)^*\otimes
{\mathcal T}\otimes K_X)^*\ =\ H^0(X,\, {\mathcal R}(E_P)^*\otimes {\mathcal O}_X(-{\mathbb D}))^*
\end{equation}
(see \eqref{e23}). Next, using \eqref{e15},
\begin{equation}\label{e27}
H^0(X,\, {\mathcal R}(E_P)^*\otimes {\mathcal O}_X(-{\mathbb D}))\ =\
H^0(X,\, (\text{ad}(E_G)/\text{ad}(E_P))\otimes {\mathcal O}_X(-{\mathbb D}))
\end{equation}
$$
\subset\,\ H^0(X,\, \text{ad}(E_G)/\text{ad}(E_P))
$$
because $\mathbb D$ is an effective divisor.

We have $H^0(X,\, \text{ad}(E_G)/\text{ad}(E_P))\,=\, 0$ because 
the holomorphic reduction of structure group $E_P\, \subset\, E_G$ is infinitesimally rigid
(see Definition \ref{def2}). Hence from \eqref{e26} and \eqref{e27} we conclude that
$$ 
H^1(X,\, {\mathcal R}(E_P)\otimes {\mathcal T}^*)\ =\ 0.
$$
This implies that $\Psi\,=\, 0$ (see \eqref{e25}). Consequently, the short exact sequence in \eqref{e20}
splits holomorphically. As noted before, this completes the proof.
\end{proof}

\begin{corollary}\label{cor1}
Let $E_P\,\subset\, E_G$ be the Harder--Narasimhan reduction of $E_G$. Let $(V,\,\phi)$ be a holomorphic
Lie algebroid over $X$ satisfying the condition that the anchor map $\phi$ is not surjective. Then the
holomorphic principal $P$--bundle $E_P$ admits a holomorphic Lie algebroid connection for $(V,\,\phi)$.
\end{corollary}

\begin{proof}
{}From Lemma \ref{lem1} we know that the Harder--Narasimhan reduction $E_P\, \subset\, E_G$
is infinitesimally rigid. So Theorem \ref{thm1} says that the holomorphic
principal $P$--bundle $E_P$ admits a holomorphic Lie algebroid connection for $(V,\,\phi)$.
\end{proof}

\begin{corollary}\label{cor2}
Take a point $x_0\, \in\, X$. Let $E_P\, \subset\, E_G$ be the Harder--Narasimhan reduction of $E_G$.
Then the holomorphic principal $P$--bundle $E_P$ admits a logarithmic connection which is nonsingular
over $X\setminus \{x_0\}$ (singularity allowed only over $x_0$).
\end{corollary}

\begin{proof}
In Corollary \ref{cor1} set $V\,=\, TX\otimes {\mathcal O}_X(-x_0)$ and set $\phi$ to
be the natural inclusion map $TX\otimes {\mathcal O}_X(-x_0)\, \hookrightarrow\, TX$.
\end{proof}

\begin{remark}\label{rem2}
For a unstable holomorphic principal $G$--bundle $E_G$ on $X$, let
$E_P\, \subset\, E_G$ be the Harder--Narasimhan reduction of $E_G$. It can be shown that $E_P$
does not admit any holomorphic connection. To see this, if $D$ is a holomorphic connection on
$E_P$, consider a nontrivial character $\chi$ of $P$ that satisfies the following two conditions:
\begin{enumerate}
\item $\chi$ is trivial on the center of $G$, and

\item $\chi$ is a nonnegative linear combination of the simple roots with respect to some Borel
subgroup contained in $P$.
\end{enumerate}
Let $L_\chi\,:=\,E_P\times^\chi {\mathbb C}$ be the holomorphic line bundle on $X$ associated to
the principal $P$--bundle $E_P$ for the character $\chi$. The degree of $L_\chi$ is nonzero
\cite[p.~695]{AAB}. But the holomorphic connection $D$ on $E_P$ induces a holomorphic connection
on the associated line bundle $L_\chi$, and the degree of any holomorphic line bundle on $X$ admitting
a holomorphic connection must be zero. In view of this contradiction we conclude that $E_P$ does
not admit any holomorphic connection.
\end{remark}

\begin{corollary}\label{cor3}
Let $E$ be a holomorphic vector bundle on $X$ of rank $r$, and let
$$
E_1\ \subset\ E_2\ \subset\ \cdots\ \subset\ E_{n-1}\ \subset\ E_n\ =\ E.
$$
be the Harder--Narasimhan filtration of $E$. Then the following two statements hold:
\begin{enumerate}
\item Take a point $x_0\, \in\, X$. The holomorphic vector bundle $E$ admits a logarithmic connection
$D$ satisfying the following two conditions:
\begin{itemize}
\item $D$ is nonsingular over $X\setminus \{x_0\}$, and

\item $D$ preserves $E_i$ for every $1\, \leq\, i\, \leq\, n$.
\end{itemize}

\item Let $(V,\, \phi)$ be a holomorphic Lie algebroid over $X$ satisfying the condition that the
anchor map $\phi$ is not surjective. Then $E$ admits a Lie algebroid connection $D$ for $(V,\, \phi)$
that preserves $E_i$ for every $1\, \leq\, i\, \leq\, n$.
\end{enumerate}
\end{corollary}
\begin{proof}
The first statement follows from Corollary \ref{cor2} by setting $G\,=\, \text{GL}(r,{\mathbb C})$ and
$P$ to be the parabolic subgroup of $\text{GL}(r,{\mathbb C})$ that preserves a filtration of subspaces
$$
F_1\ \subset\ F_2\ \subset\ \cdots\ \subset\ F_{n-1}\ \subset\ F_n\ =\ {\mathbb C}^r
$$
for which $\dim F_i\,=\, \text{rank}(E_i)$ for every $1\, \leq\, i\, \leq\, n$.

The second statement follows from Corollary \ref{cor1} by setting $G\,=\, \text{GL}(r,{\mathbb C})$ and
$P$ to be the above parabolic subgroup of $\text{GL}(r,{\mathbb C})$.
\end{proof}

%%%%%%%%%%%%%%%%%%%%%%%%%%%%%%%%%%%%%%%%%%%%%%%%%%%%%%%%%%%%%%%%%%%%%%%%%%%%%%%%%%%%%%%%%%%%%%%%%%

\section*{Acknowledgements}

\noindent The second named author is partially supported by a J. C. Bose Fellowship (JBR/2023/000003).

%%%%%%%%%%%%%%%%%%%%%%%%%%%%%%%%%%%%%%%%%%%%%%%%%%%%%%%%%%%%%%%%%%%%%%%%%%%%%%%%%%%%%%%%%%%%%%%%%


\begin{thebibliography}{ZZZZZ}

\bibitem[AO]{AO} D. Alfaya and A. Oliveira, Lie algebroid connections, twisted Higgs bundles
and motives of moduli spaces, {\it Jour. Geom. Phys.} {\bf 201} (2024), 105195.

\bibitem[ABKS]{ABKS} D. Alfaya, I. Biswas, P. Kumar and A. Singh, A criterion
for holomorphic Lie algebroid connections, {\it Jour. Alg.} {\bf 681} (2025), 343--366.

\bibitem[AAB]{AAB} B. Anchouche, H. Azad and I. Biswas, Harder--Narasimhan reduction for principal
bundles over a compact K\"ahler manifold, {\it Math. Ann.} {\bf 323} (2002), 693--712.

\bibitem[At]{At} M. F. Atiyah, Complex analytic connections in fibre bundles, {\it
Trans. Amer. Math. Soc.} {\bf 85} (1957), 181--207.

\bibitem[AB]{AB} H. Azad and I. Biswas, On holomorphic principal bundles
over a compact Riemann surface admitting a flat connection,
{\it Math. Ann.} {\bf 322} (2002), 333--346.

\bibitem[Bi]{Bi} I. Biswas, Holomorphic Lie algebroid connections on holomorphic principal bundles on
compact Riemann surfaces, {\it Jour. Math. Sci., The Univ. Tokyo} (to appear), arXiv:2511.10994.

\bibitem[Bo]{Bo} A. Borel, {\it Linear algebraic groups}, Grad. Texts in Math., 126
Springer-Verlag, New York, 1991.

\bibitem[BMRT]{BMRT} U. Bruzzo, I. Mencattini, V. Rubtsov and P. Tortella, Nonabelian Lie algebroid
extensions, {\it Int. Jour. Math.} {\bf 26} (2015) 1550040.

\bibitem[BR]{BR} U. Bruzzo and V. N. Rubtsov, Cohomology of skew-holomorphic Lie algebroids,
{\it Theoret. Math. Phys.} {\bf 165} (2010), 1596--1607.

\bibitem[Hu1]{Hu1} J. E. Humphreys, \textit{Introduction to Lie algebras and representation theory},
Grad. Texts in Math., No. 9, Springer-Verlag, New York-Heidelberg, 1972.

\bibitem[Hu2]{Hu2} J. E. Humphreys, \textit{Linear algebraic groups}, Grad. Texts in Math., No. 21,
Springer-Verlag, New York-Heidelberg, 1975.

\bibitem[To1]{To1} P. Tortella, {\it {$\Lambda$}-modules and holomorphic Lie algebroids},
PhD thesis, Scuola Internazionale Superiore di Studi Avanzati (2011).

\bibitem[To2]{To2} P. Tortella, $\Lambda$-modules and holomorphic Lie algebroid connections,
{\it Cent. Eur. J. Math.} {\bf 10} (2012), 1422--1441.

\end{thebibliography}
\end{document}